\documentclass[twoside,12pt,reqno]{amsart}
\usepackage{amsmath,amscd,amssymb}%,pstricks}%epsfig} 
\usepackage[all]{xy}
\usepackage{verbatim}
\usepackage{graphics}

\newcommand{\End}{\operatorname{End}}
\newcommand{\Endg}{\End_\g}
\newcommand{\Hom}{\operatorname{Hom}}
\newcommand{\Mod}{\operatorname{Mod}}
\newcommand{\gMod}{\g\text{-$\operatorname{Mod}$}}

\newcommand{\Id}{\operatorname{Id}}
\newcommand{\str}{\operatorname{str}}

\newcommand{\pt}{\operatorname{ptr}}
\newcommand{\sdim}{\operatorname{sdim}}

\newcommand{\qd}{\operatorname{\mathsf{d}_h}}
\newcommand{\e}{\operatorname{e}}

\newcommand{\ZZ}{\mathbb{Z}}

\newcommand{\CC}{\mathbb{C}}
\newcommand{\NN}{\mathbb{N}}
\newcommand{\sll}{\mathfrak{sl}}

\newcommand{\slmn}{\sll(m|n)}

\newcommand{\osp}{\mathfrak{osp}(2|2n)}

\newcommand{\qum}[1]{\widetilde{#1}}
\newcommand{\h}{\ensuremath{\mathfrak{h}}}
\newcommand{\g}{\ensuremath{\mathfrak{g}}}
\newcommand{\borel}{\ensuremath{\mathfrak{b}}}
\newcommand{\nil}{\ensuremath{\mathfrak{n}}}
\newcommand{\p}[1]{\ensuremath{\overline {#1}}}

\newcommand{\roots}{\Delta}
\newcommand{\weyl}{W}
\newcommand{\kzt}{A_{\g}}
\newcommand{\DJg}{\ensuremath{U_{h}^{DJ}(\g) } }

\newcommand{\IT}{\mathcal{IT}}
\newcommand{\ev}{\operatorname{ev}}
\newcommand{\coev}{\operatorname{coev}}
\newcommand{\Vect}{\operatorname{Vect}}
\newcommand{\SV}{\operatorname{SV}}

\newcommand{\ideal}{\mathcal{I}}
\newcommand{\gm}{\gMod}
\newcommand{\modf}{\text{-$\Mod$}_{fr}}

\newcommand{\wta}{\lambda}
\newcommand{\cd}{\operatorname{\mathsf{d}}}
\newcommand{\op}[1]{{#1}^-}

\newcounter{bibcount}

\newtheorem{prop}{\bf Proposition}[section] 
\newtheorem{defi}[prop]{\bf Definition}
\newtheorem{lem}[prop]{\bf Lemma} 
\newtheorem{theo}{\bf Theorem}

\newcommand{\mathsmall}[1]{\mbox{\small$#1$}}

\newcommand{\pic}[2]{
  \setlength{\unitlength}{#1}
  {\begin{array}{c} \hspace{-1.3mm}
        \raisebox{-4pt}{#2}
        \hspace{-1.9mm}\end{array}}}

\newcommand{\ptrbfa}[3]{\pic{0.7ex}{
    \begin{picture}(10,24)(0,-2)
      \qbezier(5, 3)(5, 0)(7, 0)
      \qbezier(7, 0)(9, 0)(9, 4)
      \put(2,3){\vector(0,-1){4}}
      \put(2,17){\line(0,1){4}}
      \put(9,4){\vector(0,1){12}}
      \multiput(0,3)(7,0){2}{\line(0,1){3}}
      \multiput(0,3)(0,3){2}{\line(1,0){7}}
      \multiput(0,7)(7,0){2}{\line(0,1){5}}
      \multiput(0,7)(0,5){2}{\line(1,0){7}}
      \multiput(0,13)(7,0){2}{\line(0,1){4}}
      \multiput(0,13)(0,4){2}{\line(1,0){7}}
      \put(3.5,6){\line(0,1){1}}
      \put(3.5,12){\line(0,1){1}}
      \qbezier(5, 17)(5, 20)(7, 20)
      \qbezier(7, 20)(9, 20)(9, 16)
      \put(2.5,3.6){#3}
      \put(2.5,8.7){#2}
      \put(2.5,13.7){#1}
    \end{picture}} }

\newcommand{\drawTr}{
          \qbezier(3, 3)(3, 0)(6, 0)
          \qbezier(6, 0)(9, 0)(9, 4)
          \put(9,4){\vector(0,1)3}
          \multiput(0,3)(6,0){2}{\line(0,1){5}}
          \multiput(0,3)(0,5){2}{\line(1,0){6}}
          \qbezier(3, 8)(3, 11)(6, 11)
          \qbezier(6, 11)(9, 11)(9, 7)
        }
\begin{document}
\let\co=\comment \let\endco=\endcomment
\title[A supertrace for Lie superalgebras]{An invariant supertrace for the
  category of representations of Lie superalgebras} \author{Nathan Geer}
\address{School of Mathematics\\
  Georgia Institute of Technology\\
  Atlanta, GA 30332-0160, USA} \email{geer@math.gatech.edu} \author{Bertrand
  Patureau-Mirand}
\address{LMAM, Universit\'e de Bretagne-Sud, BP 573\\
  F-56017 Vannes, France } \email{bertrand.patureau@univ-ubs.fr} \date{\today}

\begin{abstract}
  In this paper we give a re-normalization of the supertrace on the category
  of representations of Lie superalgebras of type~I, by a kind of modified
  superdimension.  The genuine superdimensions and supertraces are generically
  zero.  However, these modified superdimensions are non-zero and lead to a
  kind of supertrace which is non-trivial and invariant.  As an application we
  show that this new supertrace gives rise to a non-zero bilinear form on a
  space of invariant tensors of a Lie superalgebra of type~I.  The results of
  this paper are completely classical results in the theory of Lie
  superalgebras but surprisingly we can not prove them without using quantum
  algebra and low-dimensional topology.
\end{abstract}

\maketitle
\setcounter{tocdepth}{1}
% \tableofcontents

\section*{Introduction}
%%%%%%%%%%%%%%%%%%%%%%%%%%%%%%%%%%%%%%%%%%%%%%%%%%%%%%%%%%%%%%%%%%%%%%%%%%%%%%
%%%%%%%%%%%%%%%%%%%%%%%%%%%%%%%%%%%%%%%%%%%%%%%%%%%%%%%%%%%%%%%%%%%%%%%%%%%%%%

The theory of quantum groups and classical representation theory of Lie
algebras has been widely and productively used in low-dimensional topology.
There are fewer examples of low-dimensional topology or quantum groups being
used to produce results in the classical theory of Lie algebras.  Good
examples of such work include the theory of crystal bases (see \cite{Kash})
and the use of the Kontsevich integral to give a new proof of the
multiplicativity of the Duflo-Kirillov map $S(\g)\rightarrow U(\g)$ for
metrized Lie (super-)algebras $\g$ (see \cite{BLT}).  In this paper we use
low-dimensional topology and quantum groups to define a non-trivial kind of
supertrace on the category of representations of a Lie superalgebra of type~I.
It should be noted that the genuine supertrace is generically zero on such a
category (see Proposition \ref{P:tr0}).

In \cite{GP2,GPT}, the authors give a re-normalization of the
Reshetikhin-Turaev quantum invariants, by modified quantum dimensions.  In the
case of simple Lie algebras these modified quantum dimensions are proportional
to the genuine quantum dimensions.  For Lie superalgebras of type~I the
genuine quantum dimensions are generically zero but the modified quantum
dimensions are non-zero and lead to non-trivial link invariants.  In this case
the modified quantum dimension of a quantized module is given by an explicit
formula which is determined by the underlying Lie superalgebra module.  In
this paper we take the classical limit of the modified quantum dimension to
obtain a modified superdimension.  Then we use this modified superdimension to
re-normalize the supertrace and define a non-trivial bilinear form on a space
of invariant tensor.

Our proof that the modified supertrace is well defined and has the desired
properties is as follows.  We first formulate the desired statements at the
level of the Lie superalgebra.  Then we ``deform'' these statements to the
quantum level and use low-dimensional topology to prove these ``deformed''
statements.  Taking the classical limit we recover the original statements.
To make this proof precise we use the Etingof-Kazhdan theory of quantization.

\subsection*{Acknowledgments} 
The work of N. Geer was partially supported by NSF grant no. DMS-0706725.  He
also thanks LMAM, Universit\'e de Bretagne-Sud for the invitations during
which this work was done.

%%%%%%%%%%%%%%%%%%%%%%%%%%%%%%%%%%%%%%%%%%%%%%%%%%%%%%%%%%%%%%%%%%%%%%%%%%%%%%
%%%%%%%%%%%%%%%%%%%%%%%%%%%%%%%%%%%%%%%%%%%%%%%%%%%%%%%%%%%%%%%%%%%%%%%%%%%%%%
\section{Preliminaries}\label{s:prel}

In this section we review background material that will be used in the
following sections.

A \emph{super-space} is a $\ZZ_{2}$-graded vector space $V=V_{\p 0}\oplus
V_{\p 1}$ over $\CC$. 
We denote the parity of an homogeneous element $x\in V$
by $\p x\in \ZZ_{2}$.  We say $x$ is even (odd) if $x\in V_{\p 0}$ (resp.
$x\in V_{\p 1}$).  In the Appendix we recall some basic features and
conventions concerning the category of super-spaces.

A \emph{Lie superalgebra} is a super-space $\g=\g_{\p 0} \oplus \g_{\p 1}$
with a super-bracket $[\: , ] :\g^{\otimes 2} \rightarrow \g$ that preserves
the $\ZZ_{2}$-grading, is super-antisymmetric ($[x,y]=-(-1)^{\p x \p
  y}[y,x]$), and satisfies the super-Jacobi identity (see \cite{K}).  
 Throughout, all modules will be
$\ZZ_{2}$-graded modules (module structures which preserve the
$\ZZ_{2}$-grading, see \cite{K}).

\subsection{Lie superalgebras of type I}\label{SS:Liesuper}
In this subsection we recall notations and properties related to Lie
superalgebras of type I.

Let $\g=\g_{\p 0}\oplus \g_{\p 1}$ be a Lie superalgebra of type I, i.e.
% of type $A(m,n)$ or $C(n)$.
$\g$ is equal to $\slmn$ or $\osp$.  We will assume that $m\neq n$.  Let
$\borel$ be the distinguished Borel sub-superalgebra of $\g$.  Then $\borel$
can be written as the direct sum of a Cartan sub-superalgebra $\h$ and a
positive nilpotent sub-superalgebra $\nil_{+}$.  Moreover, $\g$ admits a
decomposition $\g=\nil_{-}\oplus \h \oplus \nil_{+}$.  Let $\weyl$ be the Weyl
group of the even part $\g_{\p0}$ of $\g$.

Let $\roots_{\p 0}^{+}$ (resp. $\roots_{\p 1}^{+}$) be the even (resp. odd)
positive roots.  Let $\rho_{\p0}$ (resp. $\rho_{\p1}$) denote the half sum of
all the even (resp. odd) positive roots.  Set $\rho= \rho_{\p0}-\rho_{\p1}$.
A positive root is called \emph{simple} if it cannot be decomposed into a sum
of two positive roots.

A Cartan matrix associated to a Lie superalgebra is a pair consisting of a $r
\times r$ matrix $A=(a_{ij})$ and a set $\tau\subset \{1,\ldots,r\}$
determining the parity of the generators.  Let $(A,\tau)$ be the Cartan matrix
arising from $\g$ and the distinguished Borel sub-superalgebra $\borel$.  Here
the set $\tau=\{s\}$ consists of only one element because of our choice of
Borel sub-algebra $\borel$. (See the appendix.)

By Proposition 1.5 of \cite{Kac78} there exists $e_{i}\in \nil_{+}$, $f_{i}\in
\nil_{-}$ and $h_{i}\in\h$ for $i=1,\ldots,r$ such that the Lie superalgebra
$\g$ is generated by $e_{i}, f_{i}, h_{i}$ where
\begin{align*} [e_{i},f_{j}]=& \delta_{ij}h_{i}, & [h_{i},h_{j}]=&0, &
  [h_{i},e_{j}]=&a_{ij}e_{j}, & [h_{i},f_{j}]=&-a_{ij}f_{j}.
\end{align*}
Note that these generators also satisfy the Serre relations and higher order
Serre type relations (see \cite{Yam94}).

There are $d_{1},\ldots,d_{r}$ in $\{\pm1,\pm2\}$ such that the matrix
$(d_{i}a_{ij})$ is symmetric.  Let $<.,.>$ be the symmetric non-degenerate
form on $\h$ determined by $<h_{i},h_{j}>=d_{j}^{-1}a_{ij}$.  This form gives
an identification of $\h$ and $\h^{*}$.  Moreover, the form $<.,.>$ induces a
$\weyl$-invariant bilinear form on $\h^{*}$, which we will also denote by
$<.,.>$.

\subsection{The category $\gm$} Modules over Lie superalgebras of type I are
different in nature than modules over semi-simple Lie algebras.  For example,
each Lie superalgebra of type I has one parameter families of modules.  Any
module in such a family has superdimension zero and so the supertrace of an
endomorphism of such a module is zero.

Let $\gm$ be the category of finite dimensional $\g$-modules (see Appendix).   
We will now describe this category in more detail.  If $U$ and $V$ are two 
$\g$-modules we denote by $\Hom_\g(U,V)$ the super-space of $\g $-module 
morphisms.  The super-space $\Hom_\g(U,V)$ should not be confused
with $\Hom_\CC(U,V)$ (where $U$ and $V$ are viewed as
super-spaces) which is naturally equipped with a $\g$-module structure.  

Let $\wta\in\h^{*}$ be a linear functional on $\h$.  Kac \cite{K} defined a
$\g$ irreducible highest weight module $V(\wta)$ of weight $\wta$ with a
highest weight vector $v_{0}$ having the property that $h.v_{0}=\wta(h)v_{0}$
for all $h\in\h$ and $\nil_{+}v_{0}=0$.  Let $a_{i}=\wta(h_{i})$.  In \cite{K}
Kac showed that $V(\wta)$ is finite-dimensional if and only if $a_{i}\in \NN$
for $i\neq s$.  Therefore, $a_{s}$ can be an arbitrary complex number.
Irreducible finite-dimensional $\g$-modules are divided into two classes:
typical and atypical. 

There are many equivalent definition for a weight module to be typical (see
\cite{Kac78}).  Here we say that $V(\wta)$ is typical if it splits in any
finite-dimensional $\g$-module (i.e. if it is a submodule or a factor-module
of a finite dimensional $\g$-module then it is a direct summand).  By
Theorem~1 of \cite{Kac78} this is equivalent to requiring that
\begin{equation}
  \label{E:typ}
  <\wta+\rho,\alpha>\neq 0
\end{equation}
for all $\alpha \in \roots_{\p 1}^{+}$.
If $V(\wta)$ is (a)typical we will say the weight $\wta$ is (a)typical.

In Section \ref{S:trace} we construct a trace on the ``ideal'' generated by
typical modules.  With this in mind let us recall some properties of these
modules.  The space of typical weights is dense in the space of weights
corresponding to finite-dimensional modules.  In particular, if $a_{i}\in \NN$
for $1\leq i\leq r$ and $i\neq s$ then there are only finitely many atypical
weights with $a_{i}=\wta(h_{i})$.  Furthermore, if $\wta$ is atypical then
$a_s=\wta(h_{s})\in\ZZ$.  Thus, the name typical is fitting.

For any object $V$ of $\gm$ whose $\ZZ_2$ grading is given by $V=V_{\p0}
\oplus V_{\p1}$ let $\sdim(V)= \dim(V_{\p0}) -\dim(V_{\p1})$ be the
superdimension of $V$.  From Proposition 2.10 of \cite{Kac78} we have that if
$V$ is a typical $\g$-module then $\sdim(V)=0$.  This vanishing can make other
mathematical objects trivial.  For example, the supertrace on endomorphisms of
a typical module and quantum invariants of links arising from Lie
superalgebras (see Proposition \ref{P:tr0} and \cite{GP2}, resp.).

Fix a typical module $V_0$.  Let $\ideal_{V_0}$ be the set of objects $V$ of
$\gm$ such that there exists an object $W$ of $\gm$ and even $\g$-linear
morphisms $\alpha:V\rightarrow V_0\otimes W$ and $\beta:V_0\otimes W
\rightarrow V$ with $\beta\circ \alpha=\Id_V$.

\begin{prop}\text{ \;}\label{P:ideal}
  \begin{enumerate}
  \item The definition of $\ideal_{V_0}$ does not depend on the choice of
    $V_0$, i.e. $\ideal_{V_0}=\ideal_{V_1}$ for any two typical modules $V_0$
    and $V_1$.
  \item The set $\ideal_{V_0}$ is an ideal in the sense that for any $V,V'\in
    \ideal_{V_0}$ and $W\in \gm$ we have $V\otimes W\in \ideal_{V_0}$ and
    $V\oplus V'\in \ideal_{V_0}$.
  \end{enumerate}
\end{prop}
We define $\ideal$ to be the set $\ideal_{V}$ where $V$ is any typical module,
which is well define by the proposition.  \\

\noindent
{\em Proof of Proposition
  \ref{P:ideal}.}  We will prove the first statement, the second follows
easily from the definition of $\ideal_{V_0}$.  
First, we have $W\in \ideal_V$ if and only if $\ideal_W \subset \ideal_V$.  We
will use this fact in the remainder of the proof.

As mentioned above irreducible finite dimensional $\g$-modules are in one to
one correspondence with $\NN^{r-1}\times \CC$.  We will denote $V_\alpha^{\bar
  c}$ as the module corresponding to $(\bar c, \alpha)\in \NN^{r-1}\times
\CC$.  Let $V_\alpha^{\bar 0}$ and $V_\beta^{\bar c}$ be typical modules.
From the character formula for typical modules we know that $V_\beta^{\bar c}$
is a submodule of $V_\alpha^{\bar 0}\otimes V_{\beta -\alpha}^{\bar c}$.
Since typical modules always split we have $V_\beta^{\bar c}\in
\ideal_{V_\alpha^{\bar 0}}$ and so $\ideal_{V_\beta^{\bar
    c}}\subset\ideal_{V_\alpha^{\bar 0}}$. 

On the other, from the discussion in the previous paragraph we have
$\Hom_\g(V_\alpha^{\bar 0}\otimes V_{\beta -\alpha}^{\bar c},V_\beta^{\bar
  c})\neq 0$, implying $\Hom_\g(V_\alpha^{\bar 0},V_\beta^{\bar c}\otimes
(V_{\beta -\alpha}^{\bar c})^*)\neq 0$.  
Therefore, as $V_\alpha^{\bar 0}$ is typical, $V_\alpha^{\bar 0}\in
\ideal_{V_\beta^{\bar c}}$ and so $\ideal_{V_\alpha^{\bar 0}}\subset\ideal_{V_\beta^{\bar
    c}}$.  \hfill \qed %\medskip

\section{A trace}\label{S:trace}
In this section we define a non-zero supertrace on $\Endg(V)$ for $V\in
\ideal$.  First, let us prove that the usual supertrace on $\Endg(V)$ is zero.

Let $V$ be a super-space and let $\{v_i\}$ be a basis of $V$ with homogeneous
vectors.  Let $\{v_i^*\}$ be the dual basis of $V^*$.  We have that
$\p{v}_i^*=\p{v}_i=\p{v}_i.\p{v}_i^*$.  Define
the supertrace on $\End_\CC(V)$ to be the function $\str_V:
\End_\CC(V)\rightarrow \CC$ given by $f\mapsto \sum_i
(-1)^{\p{v}_i}v_i^*\left(f(v_i)\right)$.  Then $\str$ has the property that if
$f\in\Hom_\CC(V,W)$ and $g\in\Hom_\CC(W,V)$ then $\str_W(f\circ
g)=(-1)^{\p f.\p g}\str_V(g\circ f)$.

Let us define the partial supertrace that is a generalization of the
supertrace.  For this, we first define the the evaluation and coevaluation
morphisms $\ev_V : V \otimes V^* \rightarrow \CC$ and $\coev_V: \CC\rightarrow
V\otimes V^*$ given by $v\otimes f\mapsto (-1)^{\p f \p v}f(v)$ and $1\mapsto
\sum_i v_i\otimes v_i^*$, respectively.
\begin{defi} \label{D:pt}
  Let $U$ and $V$ be super-spaces and $f\in\End_\CC(U\otimes V)$.
  Then we call the partial supertrace of $f$ the endomorphism
  $$\pt(f)=(\Id_U\otimes \ev_V) \circ (f\otimes\Id_{V^*}) \circ
  (\Id_U\otimes\coev_V)\in\End_\CC(U).$$
\end{defi}
For $f$ as in Definition \ref{D:pt} we have $\str_{U\otimes V}(f)=
\str_U(\pt(f))$.  In addition, if $f\in\Endg(U\otimes V)$ then
$\pt(f)\in\Endg(U)$.

Let $V$ be an element of $\ideal=\ideal_{V_0}$ and $f\in \Endg(V)$.  Choose
morphisms $\alpha : V_0\otimes W \rightarrow V$ and $\beta:V\rightarrow
V_0\otimes W$ such that $\alpha\circ \beta=\Id_V$.  Then $\pt(\beta \circ f
\circ \alpha)$ is an invariant map of $V_0$ and so $\pt(\beta \circ f \circ
\alpha)=\mathfrak{c}\Id_{V_0}$ for some $\mathfrak{c}\in \CC$.  We define the
bracket of the triple $(f,\alpha,\beta)$ to be
$<f;\alpha;\beta>=\mathfrak{c}$.
\begin{prop}\label{P:tr0}
  Let $V\in\ideal$ and $f\in \Endg(V)$ then $\str_V(f)=0$.
\end{prop}
\begin{proof}
  Using the notation above, we have
  $$\str_V(f)=\str_V(f\circ \alpha\circ \beta)=\str_{V_0\otimes W}(\beta \circ
  f\circ \alpha)=\str_{V_0}(\pt(\beta \circ f \circ \alpha)).$$ 
  But $\pt(\beta \circ f \circ
  \alpha)=<f;\alpha;\beta>\Id_{V_0}$ so
  $$\str_{V}(f)=\str_{V_0}(<f;\alpha;\beta>\Id_{V_0})
  =<f;\alpha;\beta>\sdim(V_0)=0$$ as the superdimension of $V_0$ is zero.
\end{proof}

\begin{defi}
  Let $\cd:\{\text{typical modules}\} \rightarrow \CC$ be the function defined
  by
  \begin{equation*}
    \cd(V({\wta}))=\prod_{\alpha\in\roots_{\p0}^+} \frac{<\wta
      +\rho,\alpha>}{<\rho,\alpha>}  
    \Big/ \prod_{\alpha\in\roots_{\p1}^+}<\wta +\rho,\alpha>.
  \end{equation*}
\end{defi}
Note that Equation \eqref{E:typ} implies that $\cd$ is well defined.
As an example, for $\g=\sll(n|1)$ with $n\geq2$, and for $\Lambda=(0,...,0|a)$
with $a\notin\{0,-1,\ldots,1-n\}$, we have
$\cd(V(\lambda))=\prod_{i=0}^{n-1}1/(a+i)$.

\begin{theo}\label{T:trace}
  Let $V\in\ideal$ and $f\in\Endg(V)$.  Choose a typical module $V_0$,
  $\alpha\in\Hom_\g(V_0\otimes W,V)_{\p0}$ and $\beta\in\Hom_\g(V,V_0\otimes
  W)_{\p0}$ such that $\alpha\circ \beta=\Id_V$.  Then
  $$\str'_V(f)=\cd(V_0)<f;\alpha;\beta> $$ 
  depends only on $f$, i.e. does not depend on the choice of $V_0$, $\alpha$ or
  $\beta$.  Furthermore, $\str'$ is a trace in the following sense:  for any
  $V,V'\in\ideal$ and any $\g$-module $U$,
\begin{enumerate}
\item \label{TI:main1} $\str'_V:\Endg(V) \rightarrow \CC$ is linear.
\item \label{TI:main2} $\str'_{V'}(f\circ g)=(-1)^{\p g\p f }\str'_V(g\circ
  f)$ for any $f\in\Hom_\g(V,V')$ and $g\in\Hom_\g(V',V)$,
\item \label{TI:main3} $\str'_{V\otimes U}(f\otimes g)=\str'_V(f)\,\str_{U}(g)$
  for any $f\in\Endg(V)$ and any $g\in\Endg(U)$, in particular $\str'(f\otimes
  g)=\str(g)=0$ if $U\in \ideal$.
\item \label{TI:main4bis} $\str'_{V\otimes U}(f)=\str'_V(\pt(f))$ for any
  $f\in\Endg(V\otimes U)$.
\end{enumerate}
\end{theo}
The proof of Theorem \ref{T:trace} will be given in Section \ref{S:proof}.
Let us now make a few comments about this theorem.  First, remark that
property $(\ref{TI:main4bis})$ implies property $(\ref{TI:main3})$.  Next,
property $(\ref{TI:main4bis})$ implies a kind of invariance for $\str'$.  Let
us make this statement more precise.

Let $U,U'$ be $\g$-modules and $V,V'$ be in $\ideal$.  The following spaces of
morphisms are canonically isomorphic:
\\
$\begin{array}{l}\Hom_\g(\Hom_\CC(U',V'),\Hom_\CC(U,V)) %&
  \cong\Hom_\g(U\otimes V',V\otimes U')\\
  \hspace{2ex}\cong\Hom_\g(V'\otimes U,U'\otimes V)%&
  \cong\Hom_\g(\Hom_\CC(V,U),\Hom_\CC(V',U'))\end{array}$
\\
Let $\Psi\in\Hom_\g(\Hom_\CC(U',V'),\Hom_\CC(U,V))$ and respectively
$h,\,h^\#,\,\Psi^\#$ be the corresponding morphisms in the three other
spaces. We have $h^\#=\tau\circ h \circ\tau$ where $\tau$ is the super
permutation (see Appendix).  Also, if $f\in\Hom_\CC(U',V')$ and
$g\in\Hom_\CC(V,U)$ then $\Psi(f)=\pt(h\circ (\Id_{U}\otimes f))$ and
$\Psi^\#(g)=\pt(h^\#\circ(\Id_{V'}\otimes g))$ (here we use a generalization
of the partial trace $\pt:\Hom(A\otimes C,B\otimes C)\rightarrow\Hom(A,B)$).
Thus, applying property $(\ref{TI:main4bis})$, we get that
  \begin{equation}
    \label{E:strinv}
    \str'\left(\Psi(f)\circ g\right) = (-1)^{\p \Psi.\p f}
    \str'\left(f\circ\Psi^\#( g)\right)
  \end{equation}
  Indeed, $$\begin{array}[t]{rl} \str'\left(\Psi(f)\circ g\right) &=
    \str'\left(pt(h\circ (\Id_{U}\otimes f))\circ g\right) \\&= (-1)^{\p g.\p
      f}\str'\left(pt(h\circ (g\otimes f))\right)\\&= (-1)^{\p g.\p
      f}\str'\left(h\circ (g\otimes f)\right)\\&= \str'\left(h^\#\circ
      (f\otimes g)\right)\\&= \str'\left(pt(h^\#\circ (f\otimes g))\right)\\&=
    (-1)^{\p g.\p f}\str'\left(pt(h^\#\circ (\Id_{V'}\otimes g))\circ
      f\right)\\&= (-1)^{\p \Psi.\p f}\str'\left(f\circ\Psi^\#( g)\right).
  \end{array}$$

The results of this section can be stated in the language of symmetric
monoidal category with duality or more generally ribbon categories.  We will
not make this formalism precise, however we will end this section by giving
the following graphs which we hope will shed light on the above results.  For
more details on ribbon categories see \cite{Tu}.

Here we will represent morphisms with ribbon graphs, which are read from
bottom to top.  The tensor product of two morphisms is represented by setting
the two corresponding graphs next to each other.  For example, if $f:V
\rightarrow V'$ and $g:U \rightarrow U'$ are even morphism of $\gm$ then we
represent $f$ and $f\otimes g$ by:
\begin{equation}\label{E:rep}
  \xymatrix{ 
    \ar[d]^{V'}\\
    *+[F]\txt{  $f$ } \ar[d]^{V}\\
    \: }   \put(10,-37){\text{and}}\hspace{45pt} \xymatrix{ 
    \ar[d]^{V'}\\
    *+[F]\txt{  $f$ } \ar[d]^{V}\\
    \: }\xymatrix{ 
    \ar[d]^{U'}\\
    *+[F]\txt{\;$\put(2,-6) \textit{g} $\; } \ar[d]^{U}\\
    \: }\put(0,-37){=}\hspace{2ex}\xymatrix{ 
    \ar@< 8pt>[d]^{V'}
    \ar@< -8pt>[d]_{U'}\\
    *+[F]\txt{  $f \otimes g$ } \ar@< 8pt>[d]^{V}
    \ar@< -8pt>[d]_{U}\\
    \: }
\end{equation}
respectively.  Let the graphs $\vcenter{\xymatrix@R=10pt @C=4pt{\ar@{-}
    `d^r[]^<<{V} `r^u[rr] |--{\SelectTips{cm}{}\object@{>}} [rr] & &\\ & &}}$
and $\vcenter{\xymatrix@R=10pt @C=4pt{& &\\& &\ar@{-} `u^l[] `l^d[ll]
    |--{\SelectTips{cm}{}\object@{>}} [ll]^>>{V} }}$ represent the evaluation
and coevaluation morphisms $\ev_V : V \otimes V^* \rightarrow \CC$ and
$\coev_V: \CC\rightarrow V\otimes V^*$ given by $v\otimes f\mapsto (-1)^{\p f
  \p v}f(v)$ and $1\mapsto \sum_i v_i\otimes v_i^*$, respectively.

Let $g:V\rightarrow V$ be an even invariant morphism of a $\g$-module
$V$ and let $G$ be a ribbon graph representing $g$ (as in Equation \eqref{E:rep}).
If $V$ is simple then the morphism $g$ is a scalar times the identity, which we denote by $<g>=<G>$.
  
The elements $\str_V(g)$ and $\str'_V(g)$ can be represented by
$$\str_V(g)= \left<
\pic{0.6ex}{
\begin{picture}(10,11)(1,0)
  \drawTr
  \put(2,4.5){$\mathsmall{g}$}
  \put(10,4){$\mathsmall{V}$}
\end{picture}}
\right>,$$
\begin{equation}\label{E:picstr}
  \str_V'(g)=\cd(V_0) \left< \hspace{1ex}
    \pic{0.7ex}{
      \begin{picture}(10,20)
        \put(2,3){\vector(0,-1)3}
        \qbezier(5, 3)(5, 0)(7, 0)
        \qbezier(7, 0)(9, 0)(9, 4)
        \put(9,4){\vector(0,1){12}}
        \multiput(0,3)(7,0){2}{\line(0,1){3}}
        \multiput(0,3)(0,3){2}{\line(1,0){7}}
        \put(3.5,6){\line(0,1){1}}
        \multiput(0,7)(7,0){2}{\line(0,1){5}}
        \multiput(0,7)(0,5){2}{\line(1,0){7}}
        \put(3.5,12){\line(0,1){1}}
        \multiput(0,13)(7,0){2}{\line(0,1){4}}
        \multiput(0,13)(0,4){2}{\line(1,0){7}}
        \put(2,17){\line(0,1){3}}
        \qbezier(5, 17)(5, 20)(7, 20)
        \qbezier(7, 20)(9, 20)(9, 16)
        \put(2.5,4){$\alpha$}\put(2.5,9){$g$}\put(2.5,14){$\beta$}
        \put(-1.4,.7){$\mathsmall{V_0}$}\put(9,16){$\mathsmall{W}$}
      \end{picture}}\hspace{2ex}
  \right> 
\end{equation} 
where we require $V\in\ideal$ in \eqref{E:picstr}.  When $V$ is simple the
supertrace can be rewritten as
$$\str_V(g)=\left<{\SelectTips{cm}{}\xymatrix{ \: \ar `d^r[]_{V} `r^u[] `u^l[]
      `l^d[]  []}} \right> 
\left<\vcenter{\xymatrix@R=10pt{\ar[d]^<{V}\\ *+[F]\txt{ $g$ } \ar[d]^>{V}\\
      \:} }\right>=\sdim(V)\left<\vcenter{\xymatrix@R=10pt{\ar[d]^<{V}\\
      *+[F]\txt{ $g$ } \ar[d]^>{V}\\ \:} }\right>$$
where $\sdim(V)=0$ if $V$ is typical.  Also, when $V$ is a
typical module the $\str'$ becomes
$$\str'_V(g)= \cd(V)\left<\vcenter{\xymatrix@R=10pt{\ar[d]^<{V}\\  *+[F]\txt{
        $g$ } \ar[d]^>{V}\\ \:} }\right>$$ 
Thus, the function $\cd$ can be thought of as a nonzero replacement of the
usual superdimension.  
Moreover, $\cd$ can be thought of as the classical analogue of the modified 
quantum dimensions defined in \cite{GPT}.

If $f:V\rightarrow V'$ is an even invariant morphism let $f^*: (V')^*
\rightarrow V^*$ be the ``super-transpose'' of $f$ defined in the Appendix.
We can represent can represent $f^*$ by
$$\xymatrix{ 
  \ar[d]^{V'}\\
  *+[F]\txt{  $f^*$ } \ar[d]^{V}\\
  \: }  \put(0,-37){=}\hspace{2ex}\xymatrix{  \: \\
  *+[F]{ \rotatebox{180}{{$f$}} }\ar[u]_{V^*}\\ \ar[u]^{V^{\prime *} } }
   $$
We will use the ``super-transpose'' in the next section.  
\section{Invariant tensors}
In this section we define a non-trivial bilinear form on a space of invariant
tensors of $\g$.  The standard bilinear form on $\g$ is zero on this space of
tensors.

Let $V$ be an object of $\gm$ and let $T(V)=\oplus_i T(V)_i$ be the tensor
algebra of $V$, where $T(V)_i$ is the space $V^{\otimes i}$.  Let $T(V)^\g$ be
the invariant tensors of $T(V)$.
\begin{lem}
  All invariant tensors of $T(\g)$ are even.
\end{lem}
\begin{proof}
  We will prove the lemma for $\g=\slmn$, the prove for $\osp$ is similar.  We
  can identify $\sll(m|n)$ with the Lie superalgebra of supertrace zero
  $(m+n)\times(m+n)$ matrices.  This standard representation is obtained by
  sending $e_i$ to the elementary matrix $E_{i,i+1}$, $f_i$ to $E_{i+1,i}$,
  $h_i$ to $E_{i,i}-E_{i+1,i+1}$ if $i\neq m$ and $h_m$ to
  $E_{m,m}+E_{m+1,m+1}$.  The Cartan subalgebra $\h$ with basis $(h_i)$ is
  contained in the space of diagonal matrices $X$.  The space $X^{*}$ has a
  canonical basis $(\epsilon_1,\ldots\epsilon_{m+n})$ which is dual to the
  basis formed by the matrices $E_{i,i}$. Set $\delta_{i}=\epsilon_{i+m}$,
  then $\h$ is the kernel of the supertrace
  $\str=\sum\epsilon_i-\sum\delta_j$.  Therefore, $\h^*$ is the quotient of
  $X^{*}$ by the supertrace.

  Let $\Lambda\subset \h$ be the root lattice generated by the positive roots.
  Let $f:\Lambda \rightarrow \ZZ$ be the linear function determined by
  $\epsilon_i \mapsto n$ and $\delta_j\mapsto m$ (note that $\str \mapsto 0$).
  By definition the simple positive even roots $\epsilon_i-\epsilon_j$ and
  $\delta_i-\delta_j$ map to zero and the simple positive odd roots
  $\epsilon_i -\delta_j$ map to $-(m-n)$.  Therefore, the image of $f$ is
  $(m-n)\ZZ$ and $f$ induces a linear map $\bar f: \Lambda\rightarrow
  \ZZ/2\ZZ$ given by $\alpha \mapsto \frac{f(\alpha)}{m-n}$ modulo $2$.  The
  map $\bar f$ in turn induces a map on the weight vectors of $T(\g)$
  (which we also denote by $\bar f$) that satisfies $\bar f(x\otimes y)=\bar
  f(x) + \bar f(y)$ for $x,y\in T(\g)$.
  Note that $\bar f$ gives the parity of a weight vector of $T(\g)$.  

  Let $t$ be an element of $T(\g)_k$ with weight
  $a_1\epsilon_1+\cdots+a_m\epsilon_m+b_1\delta_1+\cdots + b_n\delta_n$.  If $t$
  is in $(T(\g)_k)^\g$ then the Cartan subalgebra acts by zero and so the
  weight of $t$ is zero, i.e. $a_i=b_j=0$ for all $i$ and $j$.  But from above
  we have that parity of $t$ is equal to $f(t)=\frac{n\sum a_i+m \sum
    b_j}{m-n}$ modulo $2$, which is zero if $t$ is in $(T(\g)_k)^\g$.  Thus,
  all the invariant tensors of $T(\g)$ are even.
\end{proof}

From Propositions 2.5.3 and 2.5.5 of \cite{K} there exists a unique (up to
constant factor) non-degenerate supersymmetric invariant even bilinear form
$(,)$ on $\g$.  Let $b:\g \rightarrow \g^*$ be the isomorphism given by the
assignment $x\mapsto (x,\cdot)$.

We extend this bilinear form to $T(\g)$ by
$$(x_1x_2...x_k,x_1'x_2'...x_l')=\delta_{kl}\prod_{i=1}^k (-1)^{\sum_{i<j}
  \p{x}_j\p{x}'_i}(x_i,x'_i)$$ where $x_i,x_j'\in \g$.  Since $(,)$ is
non-degenerate on $\g$ we have that this extension is a non-degenerate bilinear
form on $T(\g)$.  Moreover, since $(,)$ is supersymmetric on $\g$ and
$(x,x')=0$ for all $x,x'\in\g$ such that $\p x\neq \p{x}'$ we have that the
extension is supersymmetric on $T(\g)$.

For $t\in(T(\g)_N)^\g\simeq\Hom_\g(\CC,T(\g)_N)$ we have
$t^*\in\Hom_\g(T(\g^*)_N,\CC)$, where $*$ is the ``super-transpose'' 
defined in the Appendix.  Using this notation the bilinear form is given by
$(t,t')=<t^*\circ b^{\otimes N}\circ t'>$.  Here and after, if $g\in
\End_\CC(\CC)$ then we will denote $<g>$ as the scalar $g(1)$.

Recall the definition of the coevaluation morphism $\coev_V$ given in 
Section \ref{S:trace}.
\begin{defi}
  For $N\in \NN $ define $$\IT_N=\{f(\coev_V(1)): f\in \Hom_{\g}(V\otimes
  V^*,\g^{\otimes n}) \text{ for some } V\in \ideal \}$$ and
  $\IT=\oplus_N\IT_N$.
\end{defi}
Let $t\in \IT_N$ and $t'\in (T(\g)_N)^\g$.  We will now show that $(t,t')$ can
be written in terms of the supertrace.  
We regard $t,t'$ as elements of $\Hom_\g(\CC,\g^{\otimes N})$.
As $t=f(\coev_V)$ for some $f\in \Hom_\g(V\otimes V^*, \g^{\otimes N})$ where
$V\in \ideal$, we have $t^*=\coev_V^* \circ f^*$ and 
$$(t,t') =<\coev_V^* \circ f^* \circ b^{\otimes N} \circ t'>.$$
The morphism $f^* \circ b^{\otimes N} \circ t'\in\Hom_\g(\CC,V^*\otimes
V)\simeq \Hom_\g(\CC,V\otimes V^*)$ can be identified with a $\g$-linear
endomorphism of $V$ which we denote by $[f^* \circ b^{\otimes N} \circ t']$.
Thus, we have $(t,t')=\str_V([f^* \circ b^{\otimes N} \circ t'])$ which is
zero by Proposition \ref{P:tr0}.  The above discussion can be summarized in the
following lemma.
\begin{lem}\label{L:formstr}
  If $t\in \IT_N$ and $t'\in (T(\g)_N)^\g$ then $(t,t')=\str_V([f^* \circ
  b^{\otimes N} \circ t'])$ which is zero.
\end{lem}

\begin{prop}\label{P:ITkenel}
  The sets $\IT_N$ are vector spaces.  Moreover, $\IT=\oplus_N\IT_N$ is a two
  sided ideal of $T(\g)^\g$ which is in the kernel of
  the restriction of $(.,.)$ to the space of invariant tensor $T(\g)^\g$.
\end{prop}
\begin{proof}
  We will first show that $\IT_N$ is a vector space.  Let $t_1,t_2\in \IT_N$
  and $\lambda\in\CC$. Then $t_i=f_i(\coev_{V_i}(1))$ for some $f_i$ and
  $V_i$.  Set $V=V_1\oplus V_2$.  Let $f:V\otimes V^*\rightarrow \g^{\otimes
    N}$ be the invariant map given by
$$f((v_1\oplus v_2)\otimes (\varphi_1\oplus \varphi_2))=f_1(v_1\otimes
\varphi_1)+\lambda f_2(v_2\otimes \varphi_2).$$  Then
$f(\coev_V(1))=t_1+\lambda t_2$.  Thus, $\IT_N$ is a vector space.   

Now we will show that $\IT$ is an ideal.  Let $t'\in (\g^{\otimes M})^{\g}$
and let $t_1$ be as above.  Let $g:V_1\otimes V_1^*\rightarrow
\g^{\otimes(M+N)}$ be the invariant map given by
$$g(v_1 \otimes \varphi_1)=t'\otimes f_1(v_1 \otimes \varphi_1).  $$
Then $g(\coev_{V_1}(1))=t'\otimes t_1$ and so $t'\otimes t_1\in \IT_{M+N}$.

The last statement of the proposition follows from Lemma \ref{L:formstr}.
\end{proof}
Next we define a bilinear form on $\IT$.  The following definition is
motivated by Lemma \ref{L:formstr} and justified by Theorem \ref{T:newbil}.
\begin{defi}
  For $t_1\in\IT_N$ and $t_2\in\IT_M$ with $t_i=f_i(\coev_{V_i})$, define
$$(t_1,t_2)'=\delta_{M,N}\str_{V_1}'\left([f_1^* \circ b^{\otimes N} \circ 
  t_2]\right)$$ 
\end{defi}
We can represent $[f_1^* \circ b^{\otimes N} \circ t_2]$ by the following
picture where $M=N=3$ for simplicity:
$$\pic{3ex}
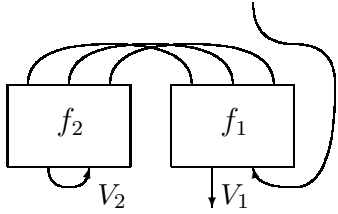
{
  \begin{picture}(8,5)(0,-1)
    % >>> box(0,0,3,2)
    \multiput(0,0)(3,0){2}{\line(0,1){2}}
    \multiput(0,0)(0,2){2}{\line(1,0){3}}
    % >>> box(4,0,7,2)
    \multiput(4,0)(3,0){2}{\line(0,1){2}}
    \multiput(4,0)(0,2){2}{\line(1,0){3}}
    % >>> curve(1,0,d,0.5,r,0.5,r,0.5,u,0.5)
    \qbezier(1, 0)(1, -0.5)(1.5, -0.5)
    \qbezier(1.5, -0.5)(2.0, -0.5)(2.0, 0.0)
    \put(1.99,-0.04){\vector(1,4){0.01}}
    % >>> curve(0.5,2,u,1,r,2,r,2,d,1)
    \qbezier(0.5, 2)(0.5, 3)(2.5, 3)
    \qbezier(2.5, 3)(4.5, 3)(4.5, 2)
    % >>> curve(1.5,2,u,1,r,2,r,2,d,1)
    \qbezier(1.5, 2)(1.5, 3)(3.5, 3)
    \qbezier(3.5, 3)(5.5, 3)(5.5, 2)
    % >>> curve(2.5,2,u,1,r,2,r,2,d,1)
    \qbezier(2.5, 2)(2.5, 3)(4.5, 3)
    \qbezier(4.5, 3)(6.5, 3)(6.5, 2)
    \put(5,0){\vector(0,-1)1}
    \put(6.01,-0.04){\vector(-1,4){0.01}}
    % >>> curve(6,0,d,0.5,r,1,r,1,u,2.5,u,1,l,1,l,1,u,1)
    \qbezier(6, 0)(6, -0.5)(7, -0.5)
    \qbezier(7, -0.5)(8, -0.5)(8, 2.0)
    \qbezier(8, 2.0)(8, 3.0)(7, 3.0)
    \qbezier(7, 3.0)(6, 3.0)(6, 4.0)
    \put(1.2,0.9){$f_2$}\put(5.2,0.9){$f_1$}
    \put(2.2,-0.9){$\mathsmall{V_2}$}\put(5.2,-0.9){$\mathsmall{V_1}$}
  \end{picture}
}
$$
It is tempting to think that the above construction could work for $t_1\in\IT$
and any $t_2\in T(\g)$ but this is false because there are examples of $t_2\in
T(\g)$ for which the above scalar depends not only of $t_1$ but also of $f_1$.

To simplify notation we will identify $\g$ and $\g^*$ using the isomorphism
$b$ but will no longer write $b$.
\begin{theo}\label{T:newbil}
  $(\cdot,\cdot)'$ is a well define symmetric bilinear form on $\IT$ satisfying
  $(G(t_1),t_2)'=(t_1,G^*(t_2))'$ for any $t_1\in\IT_M,t_2\in \IT_N$ and
  $G\in\Hom_\g(T(\g)_M,T(\g)_N)$.\\
  In particular, the symmetric group $S_N$ acts orthogonally on $\IT_N$. 
\end{theo}
\begin{proof}
  Let $t_1$ and $t_2$ be elements of $\IT_N$ with $t_i=f_i(\coev_{V_i})$.  We
  need to show that the definition of $(t_1,t_2)'$ is independent of $f_1,
  f_2, V_1,$ and $V_2$.

  Using the canonical isomorphism given in Equation \eqref{E:Hom-Tens}, we can
  identify $\Hom_\g(V_2\otimes V_2^*,V_1\otimes V_1^*)\cong
  \Hom_\g(\CC,V_1\otimes V_1^*\otimes V_2^*\otimes V_2)\cong
  \Hom_\g(\CC,V_1\otimes V_2\otimes V_1^*\otimes V_2^*)\cong \Endg(V_1\otimes
  V_2)$.
  Therefore, below we will consider $f_1^*\circ f_2$ as an element of
  $\Endg(V_1\otimes V_2)$.  Notice that for fixed $t_1=f_1(\coev_{V_1})$ the
  map $\IT_N \rightarrow \CC$ given by
  $$t\mapsto \str_{V_1}'(f_1^* \circ t)$$
  is well defined and linear.
  Then from Theorem \ref{T:trace} \eqref{TI:main4bis} we have
\begin{align*}
  \str_{V_1}'(f_1^* \circ t_2) &= \str_{V_1 \otimes V_2}'(f_1^* \circ f_2)\\
  &= \str_{V_1 \otimes V_2}'(f_2^* \circ f_1)\\
  &= \str_{V_2}'(f_2^* \circ t_1),
\end{align*}
which does not depend on $f_1$ or $V_1$.  Thus, $(\cdot,\cdot)'$ is a well
defined symmetric bilinear form.

For the last statement of the theorem we have $(G(t_1),t_2)'=\str_{V_1 \otimes
  V_2}'(f_1^*\circ G^* \circ f_2)=(t_1,G^*(t_2))'$.
\end{proof}

\section{Proof of Theorem \ref{T:trace}}\label{S:proof}
The proof of Theorem \ref{T:trace} uses quantized Lie superalgebras and
low-dimensional topology.  In particular, we have the following general plan:
(1) start with the desired statement at the level of $\gm$, (2) translate
these statements to the quantum level, (3) use properties of invariants of
ribbon graphs to prove these statements and (4) take the classical limit to
obtain the proof of the original statements.  With this in mind we will begin
this section by recalling some properties about the Drinfeld-Jimbo type
quantization of $\g$.

Let $h$ be an indeterminate and set $q=\e^{h/2}$.  We use the notation
$q^z=\e^{zh/2}$ for $z\in\CC$.  Let $\DJg$ be the Drinfeld-Jimbo type
quantization of $\g$ defined in \cite{Yam94}.  The quantization $\DJg$ is a
braided $\CC[[h]]$-Hopf superalgebra given by generators and relations.  As we
will explain now $\DJg$ is related to a quasi-Hopf superalgebra.

For each Lie algebra Drinfeld % \cite{D3}
defined a quasi-Hopf quantized universal enveloping algebra:
$$(U(\g)[[h]],\Delta_{0},\epsilon_{0},\Phi_{KZ}).$$
The morphisms $\Delta_{0}$ and $\epsilon_{0}$ are the standard coproduct and
counit of $U(\g)[[h]]$.  The element $\Phi_{KZ}$ is the KZ-associator.  Let
$\kzt$ be the analogous topologically free quasi-Hopf superalgebra (for more
details see \cite{G04A}).

Let $\DJg\modf$ ($\kzt\modf$) be the tensor category of topologically free
$\DJg$-modules (resp. $\kzt$-modules) of finite rank, i.e.  $\DJg$-modules
(resp. $\kzt$-modules) of the form $V[[h]]$ where $V$ is a finite dimensional
$\g$-module.  We say a module $V[[h]]$ in $\DJg\modf$ is typical if $V$ is a
typical $\g$-module.

In \cite{G04A} the first author proves that there exists a functor
$G:\kzt\modf\rightarrow \DJg\modf$ which is an equivalence of tensor
categories.  There is a natural tensor functor $G':\gm\rightarrow \kzt\modf$
given by $V\mapsto V[[h]]$ and $f \mapsto G'(f)$ where the action of $\g$ on
$V$ extends to an action of $U(\g)[[h]]$ on $V[[h]]$ be linearity and
$G'(f)(\sum v_ih^i)=\sum f(v_i)h^i$.  We have the following communitive
diagram of functors
\begin{equation}\label{E:functors}
  \xymatrix{ 
    \kzt\modf \ar[rr]^{G} && \DJg\modf \ar[ld]^{\text{classical limit}}\\
    &\ar[ul]^{G'} \gm &}
\end{equation}
where the down left arrow is the classical limit given by taking the limit as
$h$ goes to zero.  For any object $V$ and morphism $g$ of $\gm$ let us denote
$G\circ G'(V)$ and $G\circ G'(g)$ by $\qum{V}$ and $\qum g$, respectively.
Here the functor $G\circ G'$ composed with the classical limit is the identity
functor, i.e. $ V \equiv\qum V$ mod $h$ and $ g \equiv \qum g $ mod $h$.

In \cite{GP2} the authors define an invariant of framed colored links.  Let us
now recall the basic construction and some properties of this invariant.
Here we say that a link or more generally a tangle is \emph{colored} if each
of its components are assigned an object of $\DJg\modf$.

Let $F$ be the usual Reshetikhin-Turaev functor from the category of framed
colored tangles to the category of $\DJg\modf$.  In \cite{GP2} a function from
the set of typical $\DJg$-module to the ring $\CC[[h]][h^{-1}]$ is defined.
As remarked in \cite{GP2} this function can be multiplied by $h^{|\roots_{\p
    1}^{+}|}$ to obtain a function which takes values in $\CC[[h]]$.  Let us
denote this function by $\qd$.

\begin{lem}\label{L:qdcd}
  We have
$$\qd(\qum{V}(\wta))=h^{|\roots_{\p 1}^{+}|}\prod_{\alpha\in\roots_{\p0}^+}
\frac {q^{<\wta+\rho,\alpha>} 
  -q^{-<\wta +\rho,\alpha>}} {q^{<\rho,\alpha>} -q^{-<\rho,\alpha>}} \Big/
\prod_{\alpha\in\roots_{\p1}^+}(q^{<\wta +\rho,\alpha>}-q^{-<\wta
  +\rho,\alpha>}). $$ In particular, $\cd(V(\wta))$ is equal to
$\qd(\qum{V}(\wta))$ mod $h$.
\end{lem}
\begin{proof}
  The proof follows from the formulas for $h^{-|\roots_{\p 1}^{+}|}\qd$ given
  in the Appendix of \cite{GP2} and from the definition of $\cd$.
\end{proof}

Suppose $L$ is a framed colored link such that by cutting some component of
$L$ one obtains a framed colored $(1,1)$-tangle $T_{V(\lambda)}$ such that the
open string is colored by the deformed typical module $\qum{V}(\lambda)$ of
highest weight $\lambda$.  Then we have
$F(T_{V(\lambda)})=x.\Id_{\qum{V}(\lambda)}$, for some $x$ in $\CC[[h]]$.  Set
$<T_{V(\lambda)}>=x$.  In \cite{GP2} it is shown that the assignment
  $$L\mapsto \qd(\qum{V}(\lambda))<T_{V(\lambda)}>$$
  is a well defined colored framed link invariant denoted by $F'$.  In
  particular, $F'(L)$ is independent of $\qum{V}(\lambda)$, $T_{V(\lambda)}$
  and where $L$ is cut.

  An even morphism $f:V_1\otimes...\otimes V_n \rightarrow
  W_1\otimes...\otimes W_m$ in the category $\DJg\modf$ can be represented by
  the following box and arrows:
$$
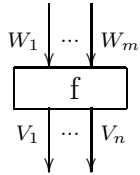
\xymatrix{ 
  \ar@< 8pt>[d]^{W_m}_{... \hspace{1pt}}
  \ar@< -8pt>[d]_{W_1}\\
  *+[F]\txt{ \: \; f \; \;} \ar@< 8pt>[d]^{V_n}_{... \hspace{1pt}}
  \ar@< -8pt>[d]_{V_1}\\
  \: }$$
Such a box is called a coupon, which we denote by
$C_{V_1,...,V_n}^{W_1,...,W_m}(f)$.  Here we will say a \emph{ribbon graph} is
a framed tangle with coupons and colors coming from the category $\DJg\modf$.
In \cite{GPT} it is shown that the construction of $F'$ can be extended to
ribbon graphs having at least one component colored by a typical
$\DJg$-module.

The invariant $F'$ can also be extended to ribbon graphs having at least one
component colored by a deformed module in $\ideal$ (see \cite{GPT}).  We will
now describe this extension in the following situation.  Let $C$ ($C'$) be a
(1,1)-tangle (resp. (2,2)-tangle) ribbon graph such that the input(s) and
output(s) are equal.  Let $L_{C}$ be the closed ribbon graph obtained from
closing the coupon $C$.  Let $T_{C'}$ be the (1,1)-tangle ribbon graph
obtained from closing right most component.  The ribbon graphs $L_C$ and
$T_{C'}$ can be represented by the following pictures
$$
L_C=\pic{0.6ex}{
        \begin{picture}(8,11)(1,0)
          \drawTr
          \put(1.6,4){$C$}
        \end{picture}} 
\qquad
T_{C'}=\pic{0.6ex}{
        \begin{picture}(8,11)(1,0)
          \qbezier(4, 3)(4, 0)(6, 0)
          \qbezier(6, 0)(9, 0)(9, 4)
          \put(9,4){\vector(0,1)3}
          \put(2,3){\vector(0,-1)3}
          \put(2,11){\vector(0,-1)3}
          \multiput(0,3)(6,0){2}{\line(0,1){5}}
          \multiput(0,3)(0,5){2}{\line(1,0){6}}
          \qbezier(4, 8)(4, 11)(6, 11)
          \qbezier(6, 11)(9, 11)(9, 7)
          \put(1.2,4){$C'$}
        \end{picture}} \,.
$$
These pictures represent respectively the trace and the partial trace of the
morphisms in the coupon.

Let $V\in\ideal$ and let $\alpha : V_0\times W \rightarrow V$ and
$\beta:V\rightarrow V_0\otimes W$ be morphisms in $\gm$ such that $\alpha\circ
\beta=\Id_V$.  Let $f\in\Endg(V)_{\p 0}$ and let $T(f;\alpha;\beta)$ be the
(1,1)-tangle ribbon graph $T_{C_{\qum V}^{\qum{V}_0\otimes \qum W}(\qum
  \beta)\circ C_{\qum V}^{\qum V}(\qum f)\circ C_{\qum{V}_0\otimes \qum
    W}^{\qum V}(\qum \alpha)}$. 
That is
$$T(f;\alpha;\beta)= {\ptrbfa{$\mathsmall{\qum\beta}$}{$\mathsmall{\qum
      f}$}{$\mathsmall{\qum\alpha}$} }$$

Then we define
 $$F'(L_{C_{\qum V}^{\qum V}(\qum f)})=\qd(\qum{V_0})<T(f;\alpha;\beta)>.$$
 In \cite{GP2,GPT} it is shown that $F'$ is well defined.
 Now we are ready to prove the main theorem of the paper.

 \vspace{5pt}

 {\em Proof of Theorem \ref{T:trace}.}  Let $V_1$ be a typical $\g$-module.
 Then we have $\ideal=\ideal_{V_0}=\ideal_{V_1}$.  Choose $\alpha_i :
 V_i\times W_i \rightarrow V$ and $\beta_i:V\rightarrow V_i\otimes W$ such
 that $\alpha_i\circ \beta_i=\Id_V$, for $i=0,1$. 
  If $f\in \Endg(V)_{\p1}$ then 
  $<f;\alpha_0;\beta_0>=<f;\alpha_1;\beta_1>=0$ as $\pt(\beta \circ f \circ
  \alpha)=<f;\alpha;\beta>\Id_{V_0}$ and $\beta \circ f \circ \alpha$ is odd.
  Therefore, we can assume that $f\in \Endg(V)_{\p0}$ (i.e. $f$ is a morphism 
  in the symmetric monoidal category $\gMod_{\p 0}$ defined in the Appendix).
 We will show that
 \begin{equation}\label{E:trwell}
   \cd(V_0)<f;\alpha_0;\beta_0>=\cd(V_1)<f;\alpha_1;\beta_1>.
 \end{equation}
 By definition of the ribbon category $\DJg\modf$ we have
 $<f;\alpha_i;\beta_i>$ is equal to $<T(f;\alpha_i;\beta_i)>$ mod $h$, for
 $i=0,1$.  Combining this with Lemma \ref{L:qdcd} we have that
 $\cd(V_i)<f;\alpha_i;\beta_i>$ is equal to
 $\qd(\qum{V}_i)<T(f;\alpha_i;\beta_i)>$ mod $h$, for $i=0,1$.  Finally, from
 \cite{GPT} we have that the extension of $F'$ to ribbon graphs is well
 define.  In particular, we have
 $\qd(\qum{V_0})<T(f;\alpha_0;\beta_0)>=\qd(\qum{V_1})<T(f;\alpha_1;\beta_1)>$.
 Thus, Equation \eqref{E:trwell} holds and $\str'_V(f)$ only depends on $f$.

 Now we prove the remaining statements of the theorem.  
 The function $\str'_V$ is linear because $F(C_{\qum V}^{\qum V}(a\qum f+b\qum
 g))=aF(C_{\qum V}^{\qum V}(\qum f))+bF(C_{\qum V}^{\qum V}(\qum g))$ for
 $f,g\in \Endg(V)_{\p0}$ and $a,b\in \CC$.  Number \eqref{TI:main3} follows
 from the property that $F'(L \sqcup L')=F'(L)F(L')$ for any two links $L$ and
 $L'$ (see \cite{GPT}).  
  The proof of Number \eqref{TI:main4bis} 
 follows from the behavior of $F'$ with respect to cabling (see \cite{GPT}).

 To prove Number \eqref{TI:main2} we need to be careful because coupons must be
 labeled by even morphisms, but the morphisms in the statement of
 \eqref{TI:main2} can be odd.  If $V$ is an object of $\gm$ then denote $\op
 V$ as the $\g$-module obtained from $V$ by taking the opposite parity.  Then
 $V$ and $\op V$ are isomorphic by an odd isomorphism $\sigma_V:V\rightarrow
 \op V$, which changes the parity.
 \begin{lem}\label{L:Proof2}
   Let $\gamma\in \End_{\DJg}(\qum{W} \otimes \qum V)_{\p0}$ and set
   $\eta=(\Id \otimes \qum{\sigma}_{V})\gamma(\Id \otimes \qum{\sigma}_{V})$.
   Then
$$F\left(T_{C_{\qum{W} \otimes \qum V}^{\qum{W} \otimes \qum
      V}(\gamma)}\right)=-F\left(T_{C_{\qum{W} \otimes \op{\qum V}}^{\qum{W}
      \otimes \op{\qum V}}(\eta)}\right).$$ 
\end{lem}
\begin{proof}
  Let $\{w_i\}_{i=1}^q$ and $\{v_j\}_{j=1}^p$ be bases of the $\g$-modules $V$
  and $W$, respectively.  Then $\{v_j\}_{j=1}^p$,
  $\{{\sigma}_{V}(v_j)\}_{j=1}^p$ and $\{w_i\}_{i=1}^q$ are bases for the
  $\DJg$-modules $\qum V$, $\op{\qum V}$ and $\qum W$, respectively.

  Let $\gamma_{ij}^{kl}$ be the elements of $\CC[[h]]$ defined by
$$\gamma(w_i\otimes v_j)=\sum_{k=1}^q \sum_{l=1}^p \gamma_{ij}^{kl} w_k\otimes
v_l.$$ 
A direct calculation shows:
\begin{align}
  F\left(T_{C_{\qum{W} \otimes \qum V}^{\qum{W} \otimes \qum
        V}(\gamma)}\right)(w_i) &= \sum_{k=1}^q \sum_{l=1}^p (-1)^{\p{v}_j}
  \gamma_{ij}^{kj} w_k, \label{E:A1}
  \\
  F\left(T_{C_{\qum{W} \otimes \op{\qum V}}^{\qum{W} \otimes \op{\qum
          V}}(\eta)}\right)(w_i) &= \sum_{k=1}^q \sum_{l,j=1}^p (-1)^{\p{w}_i
    + \p{w}_k}(-1)^{(\p1 + \p{v}_l)(\p1 + \p{v}_j)}
  \delta_{lj}\gamma_{ij}^{kl} w_k,\label{E:A2}
\end{align}
where $\delta_{lj}(-1)^{(\p1 + \p{v}_l)(\p1 + \p{v}_j)}=(-1)^{\p1 +\p{v}_j}$
and $\p{w}_i=\p{w}_k$ since $\eta$ is an even morphism.  Therefore, the right
sides of \eqref{E:A1} and \eqref{E:A2} are the negative of each other and the
lemma follows.
\end{proof}
\begin{lem}\label{L:Proof1}
  For $V\in \ideal$ and $f\in\Endg(V)_{\p 0}$ we have $$F'(L_{C_{\qum V}^{\qum
      V}(\qum f)})=-F'(L_{C_{\op{\qum V}}^{\op{\qum V}}(\qum \sigma \circ \qum
    f \circ \qum \sigma)}).$$
\end{lem}
\begin{proof}
  Let $\alpha\in\Endg(V_0\otimes W,V)_{\p0}$ and $\beta\in\Endg(V,V_0\otimes
  W)_{\p0}$ such that $\Id_V=\alpha\circ\beta$.  Then for
  $\op{\alpha}=(\Id_{V_0}\otimes\sigma_W)\circ\alpha
  \circ\sigma_V\in\Endg(V_0\otimes \op W,\op V)_{\p0}$ and
  $\op{\beta}=\sigma_V\circ\beta\circ(\Id_{V_0}\otimes\sigma_W)\in\Endg(\op
  V,V_0\otimes \op W)_{\p0}$, we have $\Id_{\op{V}}=\op\alpha\circ\op\beta$.
  Now, we also denote $\op{\qum f}= \qum \sigma_V \circ \qum f \circ \qum
  \sigma_{V}^{-1}\in \Endg(\op V)_{\p0}$ and it is convenient to give the
  following pictorial proof.
  $$\begin{array}{lcc}
    % pic1
    F'\left(\pic{0.7ex}{
        \begin{picture}(8,11)(1,0)
          \drawTr
          \put(2,4.5){$\mathsmall{\qum f}$}
        \end{picture}} \right)
    & =
    % pic2 
    F'\left(\pic{0.7ex}{
        \begin{picture}(11,24)(0,-2)
          \qbezier(3.5, 3)(3.5, 0)(7.0, 0)
          \qbezier(7.0, 0)(10.0, 0)(10.0, 4)
          \put(10,4){\vector(0,1){12}}
          \multiput(0,8)(7,0){2}{\line(0,1){3}}
          \multiput(0,8)(0,3){2}{\line(1,0){7}}
          \multiput(0,12)(7,0){2}{\line(0,1){5}}
          \multiput(0,12)(0,5){2}{\line(1,0){7}}
          \multiput(0,3)(7,0){2}{\line(0,1){4}}
          \multiput(0,3)(0,4){2}{\line(1,0){7}}
          \put(3.5,11){\line(0,1){1}}
          \put(2,7){\line(0,1){1}}
          \put(5,7){\line(0,1){1}}
          \qbezier(3.5, 17)(3.5, 20)(7.0, 20)
          \qbezier(7.0, 20)(10.0, 20)(10.0, 16)
          \put(2.5,8.6){$\mathsmall{\qum\alpha}$}
          \put(2.5,13.7){$\mathsmall{\qum f}$}
          \put(2.5,3.7){$\mathsmall{\qum\beta}$}
        \end{picture}} \right)
    & =
    % pic2.5
    F'\left(\pic{0.7ex}{
        \begin{picture}(11,24)(0,-2)
          % >>> curve(2,3,d,5,r,5,r,4,u,6)
          \qbezier(2, 3)(2, -2)(7, -2)
          \qbezier(7, -2)(11, -2)(11, 4)
          % >>> curve(5,3,d,3,r,2,r,2,u,4)
          \qbezier(5, 3)(5, 0)(7, 0)
          \qbezier(7, 0)(9, 0)(9, 4)
          \put(9,4){\vector(0,1){12}}
          \put(11,4){\vector(0,1){12}}
          % >>> box(0,3,7,6)
          \multiput(0,3)(7,0){2}{\line(0,1){3}}
          \multiput(0,3)(0,3){2}{\line(1,0){7}}
          % >>> box(0,7,7,12)
          \multiput(0,7)(7,0){2}{\line(0,1){5}}
          \multiput(0,7)(0,5){2}{\line(1,0){7}}
          % >>> box(0,13,7,17)
          \multiput(0,13)(7,0){2}{\line(0,1){4}}
          \multiput(0,13)(0,4){2}{\line(1,0){7}}
          \put(3.5,6){\line(0,1){1}}
          \put(3.5,12){\line(0,1){1}}
          % >>> curve(5,17,u,3,r,2,r,2,d,4)
          \qbezier(5, 17)(5, 20)(7, 20)
          \qbezier(7, 20)(9, 20)(9, 16)
          % >>> curve(2,17,u,5,r,5,r,4,d,6)
          \qbezier(2, 17)(2, 22)(7, 22)
          \qbezier(7, 22)(11, 22)(11, 16)
          \put(2.5,3.6){$\mathsmall{\qum\alpha}$}
          \put(2.5,8.7){$\mathsmall{\qum f}$}
          \put(2.5,13.7){$\mathsmall{\qum\beta}$}
        \end{picture}} \right)
    \\
    %pic3
    =\qd(\qum{V}_0)\left<\ptrbfa{$\mathsmall{\qum\beta}$}{$\mathsmall{\qum f}$}{$\mathsmall{\qum\alpha}$}
\right>
    &=-\qd(\qum{V}_0)\left<\ptrbfa{$\mathsmall{\op{\qum\beta}}$}{$\mathsmall{\op{\qum f}}$}{$\mathsmall{\op{\qum\alpha}}$}\right>
    &=
    -F'\left(\pic{0.7ex}{
        \begin{picture}(8,11)(1,0)
          \drawTr
          \put(2,4.5){$\mathsmall{\op{\qum f}}$}
        \end{picture}} \right)
  \end{array}$$
  where the fourth equality comes from Lemma \ref{L:Proof2}.  
\end{proof}

Now we are ready to prove Number \eqref{TI:main2}.  Let $f:V\rightarrow V'$
and $g:V'\rightarrow V$ be morphisms of $\gm$ such that $f\circ g$ is even.
If $f$ and $g$ are both even then Number \eqref{TI:main2} follows from the
fact that the closure of $C_{\qum V}^{\qum{V}'}(\qum f)\circ
C_{\qum{V}'}^{\qum V}(\qum g)$ is isotopic to closure of $C_{\qum{V}'}^{\qum
  V}(\qum g)\circ C_{\qum V}^{\qum{V}'}(\qum f)$.  If $f$ and $g$ are both odd
then Number \eqref{TI:main2} follows from the following lemma.
\begin{lem}
  If $f$ and $g$ are both odd then $$F'(L_{C_{\qum V}^{\qum V}(\qum f \circ
    \qum g)})=-F'(L_{C_{\qum V'}^{\qum V'}(\qum g \circ \qum f )}).$$
\end{lem}
\begin{proof}
  From Lemma \ref{L:Proof1} we have
  \begin{eqnarray}\label{E:proof1}
    F'(L_{C_{\qum V}^{\qum V}(f\circ g)})&= -F'(L_{C_{\op{\qum V}}^{\op{\qum
          V}}(\qum \sigma \circ \qum f \circ \qum g \circ \qum \sigma)}). 
  \end{eqnarray}
  Now since $\qum \sigma \circ \qum f$ and $\qum g \circ \qum \sigma$ are even
  we have the right side of Equation~\eqref{E:proof1} is equal to
  \begin{eqnarray*}
    -F'(L_{C_{{\qum V}}^{\op{\qum V}}(\qum \sigma \circ \qum f) \circ
      C_{\op{\qum V}}^{{\qum V}}(\qum g \circ \qum \sigma)}) 
    &=-F'(L_{ C_{\op{\qum V}}^{{\qum V}}(\qum g \circ \qum \sigma) \circ
      C_{{\qum V}}^{\op{\qum V}}(\qum \sigma \circ \qum f) })\\ 
    &=-F'(L_{ C_{{\qum V}}^{{\qum V}}(\qum g \circ \qum \sigma \qum \sigma
      \circ \qum f) })\\ 
    &=-F'(L_{ C_{{\qum V}}^{{\qum V}}(\qum g  \circ \qum f) }).
  \end{eqnarray*}
  Thus we have proved the lemma.
\end{proof}

This finishes the proof of Number \eqref{TI:main2} and the theorem. 
\hfill \qed \medskip

%%%%%%%%%%%%%%%%%%%%%%%%%%%%%%%%%%%%%%%%%%%%%%%%%%%%%%%%%%%%%%%%%%%%%%%%%%%%%%
%%%%%%%%%%%%%%%%%%%%%%%%%%%%%%%%%%%%%%%%%%%%%%%%%%%%%%%%%%%%%%%%%%%%%%%%%%%%%%
\section*{Appendix}
%B3 Maybe we should call this section a "Super-Appendix" :)
%%%%%%%%%%%%%%%%%%%%%%%%%%%%%%%%%%%%%%%%%%%%%%%%%%%%%%%%%%%%%%%%%%%%%%%%%%%%%%
%%%%%%%%%%%%%%%%%%%%%%%%%%%%%%%%%%%%%%%%%%%%%%%%%%%%%%%%%%%%%%%%%%%%%%%%%%%%%%

The theory of super-spaces follows the rule ``whenever you permute two odd
elements in an expression, put a $-$ sign''.  With this in mind, many concepts
of linear algebra have super analogs.
These analogs have new and different properties which are relevant to this
paper.
Let us discuss some of these differences.

In all the following, elements of super-spaces are generally assumed to be
homogeneous and thus their parity is well defined.  The definitions must be
generalized by linearity for non homogeneous elements.

%%%%%%%%%%%%%%%%%%%%%%%%%%%%%%%%%%%%%%%%%%%%%%%%%%%%%%%%%%%%%%%%%%%%%%%%%%%%%%
\subsection*{The category $\SV$ of super-spaces}
%%%%%%%%%%%%%%%%%%%%%%%%%%%%%%%%%%%%%%%%%%%%%%%%%%%%%%%%%%%%%%%%%%%%%%%%%%%%%%
The category $\SV$ of super-spaces is a category whose objects are
super-spaces.  The morphisms in $\SV$ between two object $U$ and $V$ denoted by
$\Hom_\CC(U,V)$ is the super-space of linear maps with the parity given by:
$$\left\{ \begin{array}{rl}
    \Hom_\CC(U,V)_{\p 0}&= \Hom_\CC(U_{\p 0},V_{\p 0}) \oplus \Hom_\CC(U_{\p
      1},V_{\p 1})\\
    \text{ and } \Hom_\CC(U,V)_{\p 1}&= \Hom_\CC(U_{\p 0},V_{\p 1})
    \oplus \Hom_\CC(U_{\p 1},V_{\p 0}).\end{array}\right.$$
This category is ``super-monoidal'' with the super version of the operator
$\otimes$ :(let us denote $\underline \otimes$ the usual tensor product in the
category $\Vect$)
\\
For two objects $U$, $V$ of $\SV$ their tensor product is the vector space
$U\underline \otimes V$ with the $\ZZ_2$-grading given by
$$\left\{ \begin{array}{rl} 
    (U\otimes V)_{\p 0}&=U_{\p 0}\underline\otimes V_{\p
      0}\,\oplus\, U_{\p 1}\underline\otimes V_{\p 1}\\\text{ and } (U\otimes
    V)_{\p 1}&=U_{\p 0}\underline\otimes V_{\p 1}\,\oplus\, U_{\p
      1}\underline\otimes V_{\p 0}.\end{array}\right.$$
and for morphisms $f\in\Hom_{\CC}(U,U')$
    and $g\in\Hom_{\CC}(V,V')$, $f\otimes g$ is given by 
$$f\otimes g=\left|\begin{array}{r}f\underline\otimes g\text{ on }U_{\p
      0}\underline\otimes V\\ (-1)^{\p g}f\underline\otimes g\text{ on }U_{\p
      1}\underline\otimes V\end{array}\right.$$
So that $(f\otimes g)(x\otimes y)=(-1)^{\p g.\p x}f(x)\otimes g(y)$.  This
tensor product realizes an isomorphism:
\begin{equation}\label{E:ConIso}
 \Hom_{\CC}(U,U')\otimes\Hom_{\CC}(V,V')\simeq\Hom_{\CC}(U\otimes V,U'\otimes
V').
\end{equation}

Let $\SV_{\p 0}$ be the subcategory of $\SV$ with the same objects but only
even morphisms (i.e. $\Hom_{\SV_{\p 0}}(U,V)=\Hom_\CC(U,V)_{\p 0}$).  The
tensor product $\otimes$ restricted to $\SV_{\p 0}$ is the usual bifunctor of
$\Vect$ with an appropriate grading on objects.  Moreover, $\SV_{\p 0}$ is a
symmetric monoidal category with symmetry isomorphisms $\tau_{U,V}:U\otimes
V\simeq V\otimes U$ given by the super permutation $\tau_{U,V}(u\otimes
v)=(-1)^{\p u.\p v}v\otimes u$.  The category $\SV$ is not a symmetric
monoidal category because in general there are morphisms $f$ and $g$ with the
property that $(\Id\otimes g)\circ(f\otimes\Id) \neq
(f\otimes\Id)\circ(\Id\otimes g)$.

For a super-space $U$, the `super-dual'' $U^*$ is defined to be the super-space 
$\Hom_\CC(U,\CC)$.  
The tensor product gives the following canonical isomorphism
$$U^*\otimes V^*=\Hom_\CC(U,\CC)\otimes\Hom_\CC(V,\CC) \simeq
\Hom_\CC(U\otimes V,\CC\otimes \CC) = (U\otimes V)^*.$$
If $f\in\Hom_\CC(U,V)$, the ``super-transpose'' of $f$ is the
linear map $f^*\in\Hom_\CC(V^*,U^*)$ given by
 $$f^*(\phi)=(-1)^{\p f.\p\phi}\phi\circ f$$
for $\phi\in V^*$.   Then, if
$f,g$ are composable morphisms of $\SV$, we have 
$$(f\circ g)^*=(-1)^{\p f.\p g}g^*\circ f^*.$$  
By convention the dual is a left dual:
\begin{itemize}
\item (left duality) $\ev_V\in\Hom_\CC(V^*\otimes V,\CC)$ is simply the
  contraction $<\phi,x>=\phi(x)$.
\item (right duality) $\ev'_V\in\Hom_\CC(V\otimes V^*,\CC)$ is given by
  $<x,\phi>=(-1)^{\p x.\p \phi} \phi(x)$
\end{itemize}
This defines a canonical isomorphism $V\rightarrow V^{**}$ when $V$ is finite
dimensional.  Again here, when restricted to $\SV_{\p 0}$ the * became a
functor, i.e. the usual contravariant duality functor with some grading
information.

%%%%%%%%%%%%%%%%%%%%%%%%%%%%%%%%%%%%%%%%%%%%%%%%%%%%%%%%%%%%%%%%%%%%%%%%%%%%%%
\subsection*{The category $\gMod$ of $\g$-modules}
%%%%%%%%%%%%%%%%%%%%%%%%%%%%%%%%%%%%%%%%%%%%%%%%%%%%%%%%%%%%%%%%%%%%%%%%%%%%%%
The universal enveloping algebra $U\g$ of $\g$ is a Hopf super-algebra, 
i.e. $U\g$ is a Hopf algebra object in $\SV_{\p 0}$.  Let $\gMod$ be
the category where objects are finite dimensional super-space $V$ with a
structure of $\g$-module (i.e. $U\g$-modules). The
morphisms of $\gMod$ are the morphisms $f$ of $\SV$ that are (``super'')
$\g$-linear:
$$\forall x\in\g,\forall v\in V,\,f(x.v)=(-1)^{\p x.\p f}x.f(v).$$

The structure of Hopf super-algebra on $U\g$ gives the tensor product of two
$\g$-modules a natural structure of $\g$-modules and the tensor product of two
$\g$-linear morphisms is $\g$-linear.  
Similarly, if $V$ is an object of $\gMod$ then the super-space $V^*$ is a
$\g$-module whose action is induced from the antipodal map of $U\g$.
$\Hom_\g(U,V)$ is canonically isomorphic to the super-space of invariant
elements of $V\otimes U^*$ and so
\begin{equation}
  \label{E:Hom-Tens}
  \Hom_\g(U,V)\cong\Hom_\g(\CC,V\otimes U^*)
\end{equation}

Let $\gMod_{\p 0}$ be the category whose objects are the objects of $\gMod$
and whose morphisms are morphisms of $\SV_{\p 0}$ which are $\g$-linear.  Then
as above $\gMod_{\p 0}$ becomes a symmetric monoidal category with duality.
Note that in general $\gMod$ is not such a category.  This is the reason we
require that the morphisms $\alpha$ and $\beta$ in the definition of $\ideal$
(see Proposition~\ref{P:ideal}) are in $\gMod_{\p 0}$.  In other words, the
proof of Theorem~\ref{T:trace} requires that we work in the category
$\gMod_{\p 0}$.

\linespread{1}

\vfill

\end{document}